\documentclass[11pt,a4paper,reqno]{amsart}

\usepackage{amsmath,amssymb,amsthm}
\usepackage{enumitem}
\usepackage{hyperref}
\usepackage{geometry}
\newtheorem{theorem}{Theorem}[section]

\newtheorem{corollary}[theorem]{Corollary}
\theoremstyle{definition}
\newtheorem{definition}[theorem]{Definition}
\newtheorem{remark}[theorem]{Remark}

\title[Dense Chains, Antichains, and Universal Partial Orders]
{Dense Chains, Antichains, and Universal Partial Orders Inside a Bounded Finite-One Degree}

\author[P.~Cintioli]{Patrizio Cintioli}

\address{Mathematics Division, School of Science and Technology, University of Camerino, Italy}
\email{patrizio.cintioli@unicam.it}

\date{}



\begin{document}

\begin{abstract}
We construct a nonrecursive set \(A\le_T\emptyset'\) and a uniformly computable
family of sets \(C_0,C_1,\dots\), all bounded finite-one equivalent to
\(A\), such that the corresponding \(1\)-degrees form a copy of the dense
linear order \((\mathbb Q,\le)\). Motivated by a recent preprint of Richter,
Stephan, and Zhang, which shows that bounded finite-one degrees can be as rigid
as a discrete \(\omega\)-chain and asks whether there are bounded finite-one
degrees consisting exactly of a dense linearly ordered set of \(1\)-degrees, we
introduce a block-density profile method for controlling one-one reducibility
inside a single bounded finite-one degree.

As further applications, in the same bounded finite-one degree we obtain an
infinite antichain of \(1\)-degrees and, more generally, an embedded copy of
every countable partial order. A single bounded finite-one degree
can already exhibit dense, incomparable, and universal order-theoretic
behaviour.

Our main technical tool is a profile theorem based on computable block-density
codings. The witness set constructed here is not \(m\)-rigid, so the
phenomena obtained in this paper arise from a mechanism different from earlier
\(m\)-rigidity-based constructions. Although our results do not settle the exact
realization problem posed by Richter, Stephan, and Zhang, we show that
density itself is not the obstruction: a single bounded finite-one degree may
already contain a copy of \((\mathbb Q,\le)\), an infinite antichain, and
embeddings of all countable partial orders.
\end{abstract}

\maketitle

\section{Introduction}

Strong refinements of many-one reducibility, such as one-one, finite-one, and
bounded finite-one reducibility, provide a natural framework for analysing the
internal structure of many-one degrees. Given a set \(A\), one may ask which
partially ordered configurations of \(1\)-degrees can occur inside a fixed
finite-one degree or bounded finite-one degree represented by \(A\).

This theme was investigated in detail by Richter, Stephan, and Zhang
\cite{RSZ2026}. On the finite-one side, they showed that nonrecursive,
nonirreducible finite-one degrees are already highly complex: every computable
partial order embeds effectively into the structure of \(1\)-degrees inside such
a finite-one degree, and hence every countable partial order embeds
noneffectively. By contrast, their bounded finite-one results reveal a much more
delicate picture. In particular, they constructed a nonrecursive bounded
finite-one degree whose internal \(1\)-degrees form a strictly ascending chain
of order type \(\omega\) \cite[Theorem~18]{RSZ2026}. The finite-one
universality phenomenon therefore does not extend automatically to the bounded
finite-one setting.

This led Richter, Stephan, and Zhang to ask which linear orders can occur
\emph{exactly} as the full structure of \(1\)-degrees inside a bounded
finite-one degree. Their Open Problem~3 \cite[Section~5]{RSZ2026} is:
\begin{quote}
\emph{Are there bounded finite-one degrees consisting exactly of a dense
linearly ordered set of \(1\)-degrees?}
\end{quote}

In their further remarks, Richter, Stephan, and Zhang go beyond Open
Problem~3 and ask more generally whether bounded finite-one degrees can exhibit
linear orders other than the order type \(\omega\) produced by their
Theorem~18. In the strong exact-realization sense made explicit there, our
results do not answer this question affirmatively, since the bounded finite-one
degree constructed here is not itself exactly linearly ordered. We show,
however, that the bounded finite-one setting already supports much richer linear
behaviour internally: a single bounded finite-one degree may contain a copy of
\((\mathbb Q,\le)\), and indeed, by
Corollary~\ref{cor:profile-universal-poset}, embedded copies of all countable
linear orders.

A different perspective on bounded finite-one degrees was developed in recent
work on \(m\)-rigidity \cite{Cintioli2026}. Recall that a set \(A\) is
\(m\)-rigid if every total computable many-one self-reduction of \(A\) is
eventually the identity. In that setting, one obtains strong negative
information about typical bounded finite-one degrees: for a measure-\(1\) and
comeager class of sets, a single bounded finite-one degree already contains an
infinite antichain of \(1\)-degrees, and therefore cannot be linearly ordered
\cite{Cintioli2026}. In the typical rigid regime, Open Problem~3 therefore has a
negative answer.

The present paper moves in the opposite direction. Rather than studying the
typical rigid case, we construct explicitly a nonrecursive set
\(A\le_T\emptyset'\) whose bounded finite-one degree exhibits a rich
internal order structure. Our main technical tool is
Theorem~\ref{thm:profile}, a profile theorem based on computable block-density
codings. Roughly speaking, each representative is assigned a computable
block-density profile recording its local preimage density over a common block
decomposition. Pointwise domination of profiles yields positive \(1\)-reductions,
while a uniform positive gap along infinitely many selected blocks allows us to
diagonalize directly against \(1\)-reductions in the opposite direction. This
produces a flexible block-geometric mechanism for prescribing comparison
patterns while keeping all representatives inside a single bounded finite-one
degree.

As consequences, we show that one can realize within a single bounded finite-one
degree a copy of the dense linear order \((\mathbb Q,\le)\), an infinite
antichain of \(1\)-degrees, and, more generally, an embedded copy of every
countable partial order; see
Corollaries~\ref{cor:profile-rationals},
\ref{cor:profile-antichain},
\ref{cor:profile-combined},
\ref{cor:profile-computable-posets}, and
\ref{cor:profile-universal-poset}. In particular, although we do not solve the
exact realization problem posed by Richter, Stephan, and Zhang, we show that
density itself is not the obstruction: a bounded finite-one degree may already
contain a copy of \((\mathbb Q,\le)\), and indeed may support universal
countable order-theoretic behaviour.

It is important that the witness constructed here lies outside the rigidity
framework of \cite{Cintioli2026}. The set \(A\) obtained in
Corollary~\ref{cor:profile-combined} is constant on a computable partition into
finite blocks and therefore admits nontrivial computable many-one
self-reductions collapsing each block to a distinguished representative; see
Remark~\ref{rem:not-m-rigid}. In particular, \(A\) is not \(m\)-rigid. Since
the class of \(m\)-rigid sets has measure \(1\) and is comeager
\cite{Cintioli2026}, our witness necessarily belongs to the null and meager
complement of the typical rigid regime.
The antichain structure arising here is therefore not a consequence of \(m\)-rigidity.
Rather, it coexists with dense linear behaviour inside a single bounded finite-one degree.

Accordingly, the contribution of this paper is not an affirmative solution to
Open Problem~3 in its exact form. The degree we construct does not consist
\emph{exactly} of a dense linear order of \(1\)-degrees, since it also contains
incomparable elements. What our results show is that the true difficulty of the
open problem lies in the word ``exactly'': bounded finite-one degrees may
already contain dense chains, infinite antichains, and even universal countable
partial orders, all inside a single explicitly constructed degree.

The paper is organised as follows. After fixing notation, we prove the Profile
Theorem, which gives a general block-density method for producing prescribed
one-one comparison patterns inside a single bounded finite-one degree.
We then derive results yielding a copy of \((\mathbb Q,\le)\), an infinite antichain,
their coexistence within the same degree, and finally universal embeddings of computable and countable partial orders.

\section{Preliminaries and Notation}

We assume familiarity with standard notions and notation from computability theory; for general background see
\cite{DH2010,Nies2009,OdifreddiCRT,OdifreddiCRT2,Rogers1967,Soare1987}.
For background more specifically on strong reducibilities and degree structures, see also
\cite{Odifreddi1981,Odifreddi1999}.
All sets considered are subsets of the natural numbers $\mathbb N = \{0,1,\dots\}$. We routinely identify a set $A \subseteq \mathbb N$ with its characteristic function, writing $A(x)=1$ if $x\in A$ and $A(x)=0$ otherwise.

We fix a standard effective enumeration $\varphi_0,\varphi_1,\dots$ of all partial computable functions. We write $\varphi_e(x)\downarrow$ to indicate that the $e$-th partial computable function converges on input $x$. We denote by $\emptyset'$ the Turing jump of the empty set.
Any set constructed by a stage construction using finitely many queries to $\emptyset'$ at each stage is computable in $\emptyset'$, and hence belongs to $\Delta^0_2$.

We now introduce the reducibilities considered in this paper, which are variants of many-one reducibility.
Recall that a set \(A\subseteq\mathbb N\) is many-one reducible to a set \(B\subseteq\mathbb N\) if there is a total computable function \(f:\mathbb N\to\mathbb N\) such that
\[
x\in A \iff f(x)\in B
\qquad\text{for all }x\in\mathbb N.
\]

\begin{definition}
Let \(A,B\subseteq\mathbb N\).
\begin{enumerate}[label=\textup{(\roman*)}]
\item We write \(A\le_1 B\) if there is a total computable injective function \(f:\mathbb N\to\mathbb N\) such that
\[
x\in A \iff f(x)\in B
\qquad\text{for all }x\in\mathbb N.
\]

\item We write \(A\le_{\rm bfo} B\) if there are a total computable function \(f:\mathbb N\to\mathbb N\) and a constant \(c\ge 1\) such that
\[
x\in A \iff f(x)\in B
\qquad\text{for all }x\in\mathbb N,
\]
and
\[
|f^{-1}(y)|\le c
\qquad\text{for all }y\in\mathbb N.
\]
\end{enumerate}
\end{definition}

We say that \(A\) and \(B\) are \(r\)-equivalent, written \(A\equiv_r B\), if both \(A\le_r B\) and \(B\le_r A\), where \(r\in\{1,{\rm bfo}\}\). The equivalence classes are the \emph{\(r\)-degrees}. We write \(A<_r B\) if \(A\le_r B\) but \(B\not\le_r A\).

By definition,
\[
A\le_1 B \implies A\le_{\rm bfo} B.
\]
Thus a bounded finite-one degree generally contains a nontrivial partially ordered structure of one-one degrees.

\section{A Profile Theorem for Block-Density Codings}\label{sec:profile}

This section isolates the combinatorial core of our method. The basic idea is to
encode each representative by a computable block-density profile measuring its
local preimage density over a common factorial block decomposition. Pointwise
comparison of profiles yields positive \(1\)-reductions, while a uniform positive
gap along infinitely many selected blocks provides enough room for direct
diagonalization against \(1\)-reductions in the opposite direction. The
following Profile Theorem packages this mechanism in a general form from which
all later applications will be derived.

\begin{definition}\label{def:profile}
A \emph{block-density profile} is a total function
\[
u:\mathbb N\to \mathbb Q\cap(1,2).
\]
A family $\{u_i\}_{i\in\mathbb N}$ of block-density profiles is \emph{uniformly computable} if the function
\[
(i,k)\longmapsto u_i(k)
\]
is computable.
\end{definition}

Informally, a block-density profile records, block by block, the local preimage density used in the coding construction.

\begin{theorem}[Profile Theorem]\label{thm:profile}
Let $\{u_i\}_{i\in\mathbb N}$ be a uniformly computable family of block-density profiles, and let
$D\subseteq\mathbb N^2$ be a computable set of ordered pairs.
Assume that for each $(i,j)\in D$ we are given, uniformly computably,
\begin{enumerate}[label=\textup{(\roman*)}]
\item a rational number $\delta_{i,j}>0$; and
\item a strictly increasing computable function
$s_{i,j}:\mathbb N\to\mathbb N$ with infinite range
$S_{i,j}=\{s_{i,j}(t):t\in\mathbb N\}$,
\end{enumerate}
such that
\begin{equation}\label{eq:profile-gap}
u_i(k)\ge u_j(k)+\delta_{i,j}
\end{equation}
for every $k\in S_{i,j}$.

\medskip

Then there exist a set $A\subseteq\mathbb N$, computable in $\emptyset'$, and a uniformly computable family of total surjective functions
\[
h_i:\mathbb N\to\mathbb N \qquad (i\in\mathbb N),
\]
such that, writing
$I_k=[a_k,a_{k+1})$ with $a_0=0$ and $a_{k+1}=a_k+k!$, the following hold:
\begin{enumerate}[label=\textup{(\alph*)}]
\item \label{itm:profile-constant}
$A$ is constant on each block $I_k$.

\item \label{itm:profile-bfo}
Each $h_i$ is bounded finite-one with global bound $2$.

\item \label{itm:profile-def-Ci}
For $C_i(x)=A(h_i(x))$, one has $C_i\equiv_{\rm bfo}A$ for every $i\in\mathbb N$.

\item \label{itm:profile-positive}
If $u_i(k)\le u_j(k)$ for every $k\in\mathbb N$, then $C_i\le_1 C_j$.

\item \label{itm:profile-negative}
If $(i,j)\in D$, then $C_i\not\le_1 C_j$.
\end{enumerate}

In particular, whenever $u_i(k)\le u_j(k)$ for all $k\in\mathbb N$ and $(j,i)\in D$, one has $C_i<_1 C_j$.
\end{theorem}

\begin{proof}
For each $k\in\mathbb N$, let
\[
a_0=0,\quad a_{k+1}=a_k+k!,\quad I_k=[a_k,a_{k+1}).
\]
Thus $\mathbb N$ is partitioned into contiguous blocks $\{I_k\}_{k\in\mathbb N}$ with $|I_k|=k!$. We construct $A\subseteq\mathbb N$ so that it is constant on each $I_k$.

For each $i,k\in\mathbb N$, define
\[
\ell_i(k)=\lfloor u_i(k)\,k!\rfloor.
\]
Since $1<u_i(k)<2$, we have
\[
k!\le \ell_i(k)<2k!
\quad\text{for all }i,k.
\]
Now let
\[
b^i_0=0,\quad b^i_{k+1}=b^i_k+\ell_i(k),\quad K_{i,k}=[b^i_k,b^i_{k+1}).
\]
Thus for each fixed $i$, the family $\{K_{i,k}\}_{k\in\mathbb N}$ is a computable partition of $\mathbb N$, and
\[
|K_{i,k}|=\ell_i(k).
\]

Define a total computable function $h_i:\mathbb N\to\mathbb N$ blockwise as follows: if
$x=b^i_k+r$ with $0\le r<\ell_i(k)$, set
\[
h_i(x)=a_k+\left\lfloor \frac{r\,k!}{\ell_i(k)}\right\rfloor.
\]
For fixed $i,k$, consider the function
\[
g_{i,k}(r)=\left\lfloor \frac{r\,k!}{\ell_i(k)}\right\rfloor
\quad(0\le r<\ell_i(k)).
\]
Since $k!\le \ell_i(k)$, we have $k!/\ell_i(k)\le 1$, so $g_{i,k}$ is nondecreasing and
\begin{equation}\label{eqn:g_i_k-increasing-by-step-1}
g_{i,k}(r+1)-g_{i,k}(r)\le 1
\quad\text{for all }r<\ell_i(k)-1.
\end{equation}
Also,
\[
g_{i,k}(0)=0
\]
and
\[
g_{i,k}(\ell_i(k)-1)
=
\left\lfloor \frac{(\ell_i(k)-1)k!}{\ell_i(k)}\right\rfloor
=
\left\lfloor k!-\frac{k!}{\ell_i(k)}\right\rfloor
=
k!-1.
\]
Since \(g_{i,k}\) is nondecreasing, integer-valued, and changes by at most \(1\)
at each step by \eqref{eqn:g_i_k-increasing-by-step-1}, its range is exactly
\(\{0,1,\dots,k!-1\}\). Since
\[
h_i(b^i_k+r)=a_k+g_{i,k}(r)
\qquad(0\le r<\ell_i(k)),
\]
it follows that
\[
h_i[K_{i,k}]
=
\{a_k+q:0\le q<k!\}
=
I_k.
\]
Since \(h_i[K_{i,k}]=I_k\) for every \(k\) and the blocks \(I_k\) partition
\(\mathbb N\), the function \(h_i\) is surjective.

Moreover, for each $q<k!$, the fiber $g_{i,k}^{-1}(q)$ consists of those $r$ such that
\[
q\le \frac{r\,k!}{\ell_i(k)}<q+1,
\]
that is,
\[
\frac{q\,\ell_i(k)}{k!}\le r<\frac{(q+1)\ell_i(k)}{k!}.
\]
This interval has length
\[
\frac{\ell_i(k)}{k!}<2,
\]
so it contains at most two integers.

Thus every fiber of the restriction
\[
h_i\!\upharpoonright K_{i,k}: K_{i,k}\to I_k
\]
has size at most $2$. Now let $y\in\mathbb N$, and let $k$ be the unique index
such that $y\in I_k$. Since $h_i(K_{i,m})\subseteq I_m$ for every $m$, and the
blocks $I_m$ are pairwise disjoint, any preimage of $y$ under $h_i$ must lie in
$K_{i,k}$. Hence
\[
|h_i^{-1}(y)|\le 2.
\]
Hence each $h_i$ is bounded finite-one with global bound $2$, proving
item~\ref{itm:profile-bfo}.

Finally, since $(i,k)\mapsto \ell_i(k)$ is computable, the endpoints
$(i,k)\mapsto b^i_k$ are uniformly computable. Therefore, given $(i,x)$, we can
effectively find the unique $k$ such that $x\in K_{i,k}$ and then compute
$h_i(x)$. So the family $\{h_i\}_{i\in\mathbb N}$ is uniformly computable.

Once $A$ has been constructed, define
\[
C_i(x)=A(h_i(x)).
\]
We first verify item~\ref{itm:profile-def-Ci}, which holds for every set $A$,
and item~\ref{itm:profile-positive}, which holds for any set $A$ that is
constant on each block $I_k$.

Since $h_i$ is total computable and bounded finite-one,
\[
C_i\le_{\rm bfo}A.
\]
Conversely, because $h_i$ is computable and surjective, the function
\[
r_i(y)=\mu x\,[h_i(x)=y]
\]
is total computable. It is injective, since two distinct values of $y$ cannot have the same least preimage. Therefore
\[
C_i(r_i(y))=A(h_i(r_i(y)))=A(y),
\]
so $A\le_1 C_i$. Hence
\[
C_i\equiv_{\rm bfo}A
\quad\text{for every }i\in\mathbb N.
\]
Thus item~\ref{itm:profile-def-Ci} holds for every set $A$.

Now suppose that $u_i(k)\le u_j(k)$ for every $k$. Then $\ell_i(k)\le \ell_j(k)$ for every $k$.
Define a total computable function $f_{i,j}:\mathbb N\to\mathbb N$ blockwise by
\[
f_{i,j}(b^i_k+r)=b^j_k+r
\quad(0\le r<\ell_i(k)).
\]
Since the partition $\{K_{i,k}\}_{k\in\mathbb N}$ is computable, given $x$ we can
find the unique $k$ and $r$ such that
\[
x=b^i_k+r,\qquad 0\le r<\ell_i(k).
\]
Because $\ell_i(k)\le \ell_j(k)$, the value
\[
f_{i,j}(x)=b^j_k+r
\]
is well-defined. On each block $K_{i,k}$ the function $f_{i,j}$ is a translation, and
distinct blocks are mapped into distinct blocks, so $f_{i,j}$ is injective.

If $x\in K_{i,k}$, then $f_{i,j}(x)\in K_{j,k}$, so
\[
h_i(x)\in I_k
\quad\text{and}\quad
h_j(f_{i,j}(x))\in I_k.
\]
Since $A$ is constant on $I_k$ and both $h_i(x)$ and $h_j(f_{i,j}(x))$ belong to
$I_k$, we have
\[
C_i(x)=A(h_i(x))=A(h_j(f_{i,j}(x)))=C_j(f_{i,j}(x)).
\]
Thus
\[
C_i\le_1 C_j.
\]
This proves that item~\ref{itm:profile-positive} holds for any $A$ that is constant on each block $I_k$.

It remains to construct $A$ so as to satisfy \ref{itm:profile-constant} and \ref{itm:profile-negative}. For every $e\in\mathbb N$ and every $(i,j)\in D$, let
\[
R_{e,i,j}:
\varphi_e \text{ is not a total one-one reduction from } C_i \text{ to } C_j.
\]
Fix an effective enumeration of all such requirements.

We build $A$ in stages using $\emptyset'$. Let $M_s$ denote the largest block
index already defined before stage $s$, with $M_0=-1$. Thus, before stage $s$,
every block $I_m$ with $m\le M_s$ has already been assigned a value. At the end
of stage $s$, every block $I_m$ with $m\le M_{s+1}$ will have been assigned a
value, and later stages will never change these values.

Suppose the $s$-th requirement is $R_{e,i,j}$. We search for a number $t\in\mathbb N$ such that, writing
\[
k=s_{i,j}(t),
\]
we have
\begin{equation}\label{eq:profile-domination}
k>M_s
\text{ and }
|K_{i,k}|>\sum_{m=0}^{k}|K_{j,m}|.
\end{equation}
This search terminates. Indeed, for every $k\in S_{i,j}$, by \eqref{eq:profile-gap},
\begin{equation}\label{eqn:K_i-K_j}
|K_{i,k}|-|K_{j,k}|
=
\lfloor u_i(k)k!\rfloor-\lfloor u_j(k)k!\rfloor
\ge \delta_{i,j}k!-1.
\end{equation}
On the other hand, for every $m$,
\[
|K_{j,m}|<2m!,
\]
so
\[
\sum_{m=0}^{k-1}|K_{j,m}|<2\sum_{m=0}^{k-1}m!.
\]
Since for $k\ge 1$,
\[
\sum_{m=0}^{k-1}m!\le (k-1)!\sum_{r=0}^{k-1}\frac1{r!}<3(k-1)!,
\]
we obtain
\begin{equation}\label{eqn:sum_less_than_6(k-1)!}
\sum_{m=0}^{k-1}|K_{j,m}|<6(k-1)!.
\end{equation}
Because
\[
\delta_{i,j}k!=\delta_{i,j}k(k-1)!,
\]
the quantity $\delta_{i,j}k!-1$ eventually dominates $6(k-1)!$, so by \eqref{eqn:K_i-K_j} and \eqref{eqn:sum_less_than_6(k-1)!} we obtain
\[
|K_{i,k}|-|K_{j,k}|>\sum_{m<k}|K_{j,m}|
\]
which implies
\[
|K_{i,k}|>|K_{j,k}|+\sum_{m<k}|K_{j,m}|
=\sum_{m\le k}|K_{j,m}|.
\]

Hence, for all sufficiently large $k\in S_{i,j}$, we have
\[
|K_{i,k}|>\sum_{m\le k}|K_{j,m}|,
\]
and also $k>M_s$.
Equation~\eqref{eq:profile-domination} therefore holds for all sufficiently large $k\in S_{i,j}$.

Since $S_{i,j}$ is infinite and computable, the search for $t$ terminates.

Having chosen such a $k$, we act as follows.

\begin{enumerate}[label=\textup{[Step \arabic*]}]
\item Use $\emptyset'$ to decide whether $\varphi_e(x)\downarrow$ for every
$x\in K_{i,k}$. If not, then $\varphi_e$ is not total, so $R_{e,i,j}$ is
already satisfied. In this case, for every previously undefined block $I_m$
with $M_s<m\le k$, define
\[
A(z)=0
\quad\text{for all }z\in I_m,
\]
and set $M_{s+1}=k$.

\item Suppose now that $\varphi_e(x)\downarrow$ for every $x\in K_{i,k}$. Since
$K_{i,k}$ is finite, we can compute all values $\varphi_e(x)$ for
$x\in K_{i,k}$. If these values are not pairwise distinct, then $\varphi_e$ is
not injective, so $R_{e,i,j}$ is satisfied. In that case, for every previously
undefined block $I_m$ with $M_s<m\le k$, define
\[
A(z)=0
\quad\text{for all }z\in I_m,
\]
and set $M_{s+1}=k$.

\item Suppose finally that $\varphi_e$ is total and injective on $K_{i,k}$.
By \eqref{eq:profile-domination},
\[
|K_{i,k}|>\left|\bigcup_{m\le k}K_{j,m}\right|.
\]
Since $\varphi_e$ is injective on $K_{i,k}$,
\[
|\varphi_e[K_{i,k}]|=|K_{i,k}|,
\]
so $\varphi_e[K_{i,k}]$ cannot be contained in $\bigcup_{m\le k}K_{j,m}$.

Hence there exists some $x\in K_{i,k}$ such that
\[
y=\varphi_e(x)\in K_{j,\ell}
\]
for some $\ell>k$.
Fix the least such $x$, and let $\ell$ be the corresponding index.
Define
\[
A(z)=1 \quad\text{for all }z\in I_k,
\quad
A(z)=0 \quad\text{for all }z\in I_\ell,
\]
and for every previously undefined block $I_m$ with $M_s<m<\ell$ and $m\neq k$,
define
\[
A(z)=0
\quad\text{for all }z\in I_m.
\]
Finally let $M_{s+1}=\ell$.
\end{enumerate}

We now verify the construction by checking that every requirement $R_{e,i,j}$ with $(i,j)\in D$ is satisfied.
Fix $e\in\mathbb{N}$ and $(i,j)\in D$.
Let $s$ be the stage at which the requirement $R_{e,i,j}$ is considered.
In Steps~1 and~2, the requirement $R_{e,i,j}$ is satisfied because
$\varphi_e$ is not total in the first case and not injective in the second.
In Step~3, with $x\in K_{i,k}$ and
$y=\varphi_e(x)\in K_{j,\ell}$, we have
\[
h_i(x)\in I_k
\quad\text{and}\quad
h_j(y)\in I_\ell.
\]
Since $A$ is constant with value $1$ on $I_k$ and constant with value $0$ on $I_\ell$, we obtain
\[
C_i(x)=A(h_i(x))=1
\quad\text{and}\quad
C_j(\varphi_e(x))=C_j(y)=A(h_j(y))=0.
\]
Hence
\[
C_i(x)\neq C_j(\varphi_e(x)),
\]
so $\varphi_e$ is not a one-one reduction from $C_i$ to $C_j$.
Since $e$ and $(i,j)\in D$ were arbitrary, every requirement $R_{e,i,j}$ is
satisfied. Therefore item~\ref{itm:profile-negative} holds.

At every stage, values are assigned only to blocks $I_m$ with $m>M_s$, so there
is no injury. Moreover, the sequence $(M_s)_{s\in\mathbb{N}}$ is strictly increasing, and at the end of stage $s$
all blocks $I_m$ with $m\le M_{s+1}$ have been defined. Hence every block $I_m$
is defined at some finite stage. Thus $A$ is total and constant on each block,
which gives item~\ref{itm:profile-constant}.

To compute $A(n)$ from $\emptyset'$, first compute the unique $m$ such that
$n\in I_m$. Then, using $\emptyset'$, simulate the construction stage by stage
until the end of some stage $s$ with $M_{s+1}\ge m$. By the no-injury property,
the value assigned to the block $I_m$ has then stabilized permanently, so we can
output $A(n)$. Therefore $A\le_T\emptyset'$. This completes the proof.

\end{proof}

\begin{remark}\label{rem:A-below-Ci}
In the situation of Theorem~\ref{thm:profile}, the $1$-degree of $A$ is a common lower bound of the family $\{\deg_1(C_i)\}_{i\in\mathbb{N}}$.
Indeed, for every $i$, the
least-preimage function
\[
r_i(y)=\mu x\,[h_i(x)=y]
\]
is total computable and injective, and satisfies
\[
A(y)=C_i(r_i(y)).
\]
Hence
\[
A\le_1 C_i
\qquad\text{for every }i\in\mathbb N.
\]
\end{remark}

The next two corollaries illustrate the two basic regimes of the profile method.
Constant profiles yield monotone comparison patterns and hence a dense rational
chain, whereas oscillating profiles create local dominance in opposite
directions and therefore force pairwise incomparability.

\begin{corollary}[Constant profiles]\label{cor:profile-rationals}
Fix an injective computable enumeration
\[
q_0,q_1,\dots
\]
of $\mathbb Q\cap(1,2)$, and define
\[
u_i(k)=q_i
\quad\text{for all }i,k\in\mathbb N.
\]
Let
\[
D=\{(i,j)\in\mathbb N^2:q_i>q_j\}.
\]
For $(i,j)\in D$, define
\[
\delta_{i,j}=q_i-q_j
\quad\text{and}\quad
s_{i,j}(t)=t.
\]
Then Theorem~\ref{thm:profile} yields a set $A\le_T\emptyset'$ and sets
\[
B_{q_i}=C_i
\]
such that
\[
B_{q_i}\equiv_{\rm bfo}A
\quad\text{for every }i\in\mathbb N,
\]
and
\[
B_{q_i}<_1 B_{q_j}
\quad\text{if and only if}\quad
q_i<q_j.
\]
Hence the $1$-degrees inside the bounded finite-one degree of $A$ contain a copy of $(\mathbb Q,\le)$.
\end{corollary}

\begin{proof}
The hypotheses of Theorem~\ref{thm:profile} are satisfied: the family
$u_i(k)=q_i$ is uniformly computable, the set
\[
D=\{(i,j)\in\mathbb N^2:q_i>q_j\}
\]
is computable, for each $(i,j)\in D$ the rational
\[
\delta_{i,j}=q_i-q_j>0
\]
is uniformly computable, and
\[
s_{i,j}(t)=t
\]
is strictly increasing, computable, and has infinite range. Moreover, for every
$(i,j)\in D$ and every $k$,
\[
u_i(k)=q_i=q_j+(q_i-q_j)=u_j(k)+\delta_{i,j}.
\]
It follows from Theorem~\ref{thm:profile} that there exist a set $A\le_T\emptyset'$ and sets
$C_i$ such that $C_i\equiv_{\rm bfo}A$ for every $i$. Define
\[
B_{q_i}=C_i.
\]

If $q_i<q_j$, then $u_i(k)\le u_j(k)$ for every $k$, so
\[
B_{q_i}\le_1 B_{q_j}
\]
by item~\ref{itm:profile-positive} of Theorem~\ref{thm:profile}. Since
$q_j>q_i$, we have $(j,i)\in D$, and therefore
\[
B_{q_j}\not\le_1 B_{q_i}
\]
by item~\ref{itm:profile-negative}. Hence
\[
B_{q_i}<_1 B_{q_j}.
\]

Conversely, suppose that
\[
B_{q_i}<_1 B_{q_j}.
\]
Then $B_{q_i}\le_1 B_{q_j}$. If $q_i>q_j$, then $(i,j)\in D$, and
item~\ref{itm:profile-negative} would give
\[
B_{q_i}\not\le_1 B_{q_j},
\]
a contradiction. If $q_i=q_j$, then since the enumeration $(q_i)_{i\in\mathbb N}$
is injective, we must have $i=j$, hence $B_{q_i}=B_{q_j}$, contradicting
$B_{q_i}<_1 B_{q_j}$. Therefore $q_i<q_j$.

Thus
\[
B_{q_i}<_1 B_{q_j}
\quad\text{if and only if}\quad
q_i<q_j.
\]
Moreover,
\[
B_{q_i}\le_1 B_{q_j}\quad\Longleftrightarrow\quad q_i\le q_j.
\]
Indeed, the implication ``$\Rightarrow$'' follows from item~\ref{itm:profile-negative}, since \(q_i>q_j\) would imply \((i,j)\in D\), while ``$\Leftarrow$'' follows from item~\ref{itm:profile-positive} when \(q_i<q_j\), and is trivial when \(q_i=q_j\). Therefore the function
\[
q\longmapsto \deg_1(B_q)
\]
is an order-embedding of \((\mathbb Q\cap(1,2),\le)\) into the poset of \(1\)-degrees inside the bounded finite-one degree of \(A\).

Since each $B_{q_i}\equiv_{\rm bfo}A$, the corresponding $1$-degrees all lie inside
the bounded finite-one degree of $A$.
The corresponding $1$-degrees therefore form a copy of
$(\mathbb Q\cap(1,2),\le)$, and hence a copy of $(\mathbb Q,\le)$.

\end{proof}

The previous corollary used constant profiles to obtain a dense chain. We now
turn to the complementary phenomenon: by letting the profiles oscillate between
two rational densities on pairwise disjoint computable supports, we force
mutual incomparability and thereby obtain an infinite antichain of \(1\)-degrees.

\begin{corollary}[Oscillating profiles]\label{cor:profile-antichain}
Fix rationals $1<\alpha<\beta<2$, and for each $n\in\mathbb N$ let
\[
Z_n=\{2^n(2t+1):t\in\mathbb N\}.
\]
Then the sets $Z_n$ are pairwise disjoint, infinite, and uniformly computable.
Define, for each $n,k\in\mathbb N$,
\[
u_n(k)=
\begin{cases}
\beta & \text{if }k\in Z_n,\\
\alpha & \text{if }k\notin Z_n.
\end{cases}
\]
Let
\[
D=\{(n,m)\in\mathbb N^2:n\neq m\}.
\]
For $n\neq m$, define
\[
\delta_{n,m}=\beta-\alpha
\quad\text{and}\quad
s_{n,m}(t)=2^n(2t+1).
\]
Then Theorem~\ref{thm:profile} yields a set $A\le_T\emptyset'$ and sets
\[
W_n=C_n
\]
such that
\[
W_n\equiv_{\rm bfo}A
\quad\text{for every }n\in\mathbb N,
\]
and
\[
W_n\not\le_1 W_m
\quad\text{whenever }n\neq m.
\]
Hence $\{\deg_1(W_n)\}_{n\in\mathbb N}$ is an infinite antichain of $1$-degrees inside a single bounded finite-one degree.
\end{corollary}

\begin{proof}

For each $n\in\mathbb N$,
\[
Z_n=\{2^n(2t+1):t\in\mathbb N\}
\]
is infinite, and the sets $Z_n$ are pairwise disjoint. The family $\{Z_n\}_{n\in\mathbb N}$ is
uniformly computable, since
\[
k\in Z_n \iff k>0,\ 2^n\mid k,\ \text{and }2^{n+1}\nmid k.
\]
It follows that the family of profiles $\{u_n\}_{n\in\mathbb N}$ is uniformly
computable. Moreover,
\[
D=\{(n,m)\in\mathbb N^2:n\neq m\}
\]
is computable. For each $(n,m)\in D$, the rational
\[
\delta_{n,m}=\beta-\alpha>0
\]
is uniformly computable, and
\[
s_{n,m}(t)=2^n(2t+1)
\]
is a strictly increasing computable function with infinite range
\[
S_{n,m}=Z_n.
\]

Now fix $(n,m)\in D$, and let $k\in S_{n,m}=Z_n$. Since $n\neq m$ and the sets
$Z_n$ are pairwise disjoint, we have $k\notin Z_m$. Hence
\[
u_n(k)=\beta
\quad\text{and}\quad
u_m(k)=\alpha.
\]
Therefore
\[
u_n(k)=\beta=\alpha+(\beta-\alpha)=u_m(k)+\delta_{n,m},
\]
so condition~\eqref{eq:profile-gap} holds.

Applying Theorem~\ref{thm:profile} and writing
\[
W_n:=C_n \qquad(n\in\mathbb N),
\]
we obtain a set \(A\le_T\emptyset'\) and sets \(W_n\) such that
\[
W_n\equiv_{\rm bfo}A
\quad\text{for every }n\in\mathbb N,
\]
and
\[
W_n\not\le_1 W_m
\quad\text{whenever }n\neq m.
\]
Thus the sets \(W_n\) are pairwise incomparable under \(\le_1\), so the
corresponding \(1\)-degrees
\[
\{\deg_1(W_n):n\in\mathbb N\}
\]
form an infinite antichain. Since each \(W_n\equiv_{\rm bfo}A\), all these
\(1\)-degrees lie inside the single bounded finite-one degree of \(A\).

\end{proof}

The previous two corollaries isolate the two basic effects of the profile
method: constant profiles yield a dense chain, while oscillating profiles force
pairwise incomparability. The next corollary shows that these are not separate
phenomena attached to different witnesses. By combining the two profile
patterns in a single construction, we obtain one bounded finite-one degree that
simultaneously contains both a copy of \((\mathbb Q,\le)\) and an infinite
antichain of \(1\)-degrees.

\begin{corollary}[One degree containing both structures]\label{cor:profile-combined}
There exist a nonrecursive set $A\le_T\emptyset'$, a family $\{B_q\}_{q\in\mathbb Q\cap(1,2)}$, and a family $\{W_n\}_{n\in\mathbb N}$ such that
\begin{enumerate}[label=\textup{(\roman*)}]
\item $B_q\equiv_{\rm bfo}A$ for every $q\in\mathbb Q\cap(1,2)$;
\item $W_n\equiv_{\rm bfo}A$ for every $n\in\mathbb N$;
\item $B_p<_1 B_q$ whenever $p,q\in\mathbb Q\cap(1,2)$ and $p<q$;
\item $W_n\not\le_1 W_m$ whenever $n\neq m$.
\end{enumerate}
Consequently, the bounded finite-one degree of $A$ contains both a copy of $(\mathbb Q,\le)$ and an infinite antichain of $1$-degrees.
\end{corollary}

\begin{proof}
Fix an injective computable enumeration
\[
q_0,q_1,\dots
\]
of $\mathbb Q\cap(1,2)$, and fix rationals $1<\alpha<\beta<2$.
For each $n\in\mathbb N$, let
\[
Z_n=\{2^n(2t+1):t\in\mathbb N\}.
\]

Define a uniformly computable family of profiles $\{u_r\}_{r\in\mathbb N}$ by
\[
u_{2i}(k)=q_i
\quad(i,k\in\mathbb N),
\]
and
\[
u_{2n+1}(k)=
\begin{cases}
\beta & \text{if }k\in Z_n,\\
\alpha & \text{if }k\notin Z_n
\end{cases}
\quad(n,k\in\mathbb N).
\]

Let $D$ consist of the following pairs:
\begin{enumerate}[label=\textup{(\alph*)}]
\item every pair $(2i,2j)$ such that $q_i>q_j$;
\item every pair $(2n+1,2m+1)$ such that $n\neq m$.
\end{enumerate}
For pairs of type \textup{(a)}, define
\[
\delta_{2i,2j}=q_i-q_j
\quad\text{and}\quad
s_{2i,2j}(t)=t.
\]
For pairs of type \textup{(b)}, define
\[
\delta_{2n+1,2m+1}=\beta-\alpha
\quad\text{and}\quad
s_{2n+1,2m+1}(t)=2^n(2t+1).
\]
No relation is imposed between even and odd indices.
We claim that the hypotheses of Theorem~\ref{thm:profile} are satisfied.
The family
$\{u_r\}_{r\in\mathbb N}$ is uniformly computable: for even indices,
$u_{2i}(k)=q_i$, and for odd indices, $u_{2n+1}(k)$ is determined by the
uniformly computable predicate $k\in Z_n$. The set $D$ is computable. For
pairs of type \textup{(a)}, the rational $\delta_{2i,2j}=q_i-q_j>0$ is
uniformly computable, $s_{2i,2j}(t)=t$ is strictly increasing and computable
with infinite range, and for every $k$,
\[
u_{2i}(k)=q_i=q_j+(q_i-q_j)=u_{2j}(k)+\delta_{2i,2j}.
\]
For pairs of type \textup{(b)}, the rational
$\delta_{2n+1,2m+1}=\beta-\alpha>0$ is uniformly computable, the function
$s_{2n+1,2m+1}(t)=2^n(2t+1)$ is strictly increasing and computable with
infinite range $Z_n$, and if $k\in Z_n$, then $k\notin Z_m$, so
\[
u_{2n+1}(k)=\beta=\alpha+(\beta-\alpha)=u_{2m+1}(k)+\delta_{2n+1,2m+1}.
\]
Thus condition~\eqref{eq:profile-gap} holds for every pair in $D$.
Applying Theorem~\ref{thm:profile} and writing
\[
B_{q_i}=C_{2i},
\quad
W_n=C_{2n+1},
\]
we obtain
\[
B_{q_i}\equiv_{\rm bfo}A
\quad\text{and}\quad
W_n\equiv_{\rm bfo}A
\]
for all $i,n$. By items~\ref{itm:profile-positive} and~\ref{itm:profile-negative} of Theorem~\ref{thm:profile},
\[
B_{q_i}<_1 B_{q_j}
\quad\text{whenever }q_i<q_j,
\]
and
\[
W_n\not\le_1 W_m
\quad\text{whenever }n\neq m.
\]
Thus the bounded finite-one degree of $A$ contains both a copy of $(\mathbb Q,\le)$ and an infinite antichain of $1$-degrees.

It remains to show that the set $A$ produced by the construction in the proof
of Theorem~\ref{thm:profile} is nonrecursive. Fix an index $e_{\mathrm{id}}$
for the identity function. For every pair $i,j\in\mathbb N$ with $q_i>q_j$,
the requirement $R_{e_{\mathrm{id}},\,2i,\,2j}$ occurs in that construction.
Its vacuous cases are impossible, since the identity is total and injective on
every finite set. Hence the construction reaches the diagonalization case, and
for fresh blocks $I_k<I_\ell$ it assigns
\[
A(z)=1 \text{ for all } z\in I_k,
\quad
A(z)=0 \text{ for all } z\in I_\ell.
\]
Since there are infinitely many pairs $(i,j)$ with $q_i>q_j$, this happens
infinitely often. Therefore $A$ has infinitely many $1$-blocks and infinitely
many $0$-blocks, so $A$ is infinite and coinfinite.
For each \(i\) and each \(k\), the block \(K_{2i,k}\) is nonempty and
\[
h_{2i}[K_{2i,k}]=I_k.
\]
Since \(A\) is constant on \(I_k\), it follows that \(B_{q_i}(x)=A(h_{2i}(x))\)
is constant on \(K_{2i,k}\), with the same value that \(A\) takes on \(I_k\).
As \(A\) has infinitely many \(1\)-blocks and infinitely many \(0\)-blocks, the
set \(B_{q_i}\) also has infinitely many \(1\)-blocks and infinitely many
\(0\)-blocks. Hence \(B_{q_i}\) is infinite and coinfinite.

If $A$ were recursive, then every $B_{q_i}$ would be recursive.
But any two infinite coinfinite recursive sets are one-one equivalent.
This contradicts the fact that
\[
B_{q_i}<_1 B_{q_j}
\quad\text{whenever }q_i<q_j.
\]
Hence $A$ is nonrecursive.

\end{proof}

\begin{remark}
By Remark~\ref{rem:A-below-Ci}, the \(1\)-degree of \(A\) is a common lower bound
of the explicitly constructed degrees. More precisely, in
Corollaries~\ref{cor:profile-rationals} and~\ref{cor:profile-combined} one has
\[
A<_1 B_q
\qquad(q\in\mathbb Q\cap(1,2)),
\]
while in Corollaries~\ref{cor:profile-antichain}
and~\ref{cor:profile-combined} one has
\[
A<_1 W_n
\qquad(n\in\mathbb N).
\]
Indeed, if \(B_q\le_1 A\), choose \(p<q\). Then
\[
A\le_1 B_p<_1 B_q\le_1 A,
\]
a contradiction. Similarly, if \(W_n\le_1 A\), then for any \(m\neq n\),
\[
W_n\le_1 A\le_1 W_m,
\]
contradicting \(W_n\not\le_1 W_m\).
Thus both the copy of \((\mathbb Q,\le)\) and the infinite antichain constructed
in Corollary~\ref{cor:profile-combined} lie strictly above the \(1\)-degree of \(A\).
\end{remark}

\begin{remark}\label{rem:scope}
Corollary~\ref{cor:profile-combined} concerns the specific bounded finite-one degree produced by the combined family of constant and oscillating profiles. It does not assert that every bounded finite-one degree containing a dense chain of $1$-degrees must also contain an infinite antichain.
\end{remark}

\begin{remark}\label{rem:not-m-rigid}
The witness $A$ from Corollary~\ref{cor:profile-combined} is not $m$-rigid. Indeed, let $I_k=[a_k,a_{k+1})$ be the computable target blocks from the construction, and recall that $A$ is constant on each $I_k$. Define $f:\mathbb N\to\mathbb N$ by
\[
f(x)=a_k
\quad\text{whenever }x\in I_k.
\]
Then $f$ is total computable, and for every $x\in I_k$ one has
\[
x\in A \iff a_k\in A \iff f(x)\in A.
\]
Thus $f$ is a computable $m$-autoreduction of $A$. However, for every $k\ge 2$ and every $x\in I_k\setminus\{a_k\}$, we have $f(x)\neq x$. Hence $f$ is not eventually the identity, so $A$ is not $m$-rigid.
\end{remark}

\subsection{Universal order embeddings}

The preceding corollaries show that the profile method can realize both dense
chains and infinite antichains inside a single bounded finite-one degree. We now
turn to its most general consequence. For each element \(i\), we encode its
principal lower set
\[
\downarrow i=\{n\in\mathbb N : n\preceq i\}
\]
into a computable block profile. In this way, pointwise domination of profiles
mirrors inclusion of lower sets, and hence the original order relation.

The next corollary may be viewed as the universal form of the oscillating-profile
construction. The pairwise disjoint computable sets \(Z_n\) serve as independent
markers for the elements \(n\) of the underlying partial order, and the profile
\(u_i\) records exactly which markers belong to \(\downarrow i\).

\begin{corollary}[Computable partial orders]\label{cor:profile-computable-posets}
Let $\preceq$ be a computable partial order on $\mathbb N$. Then there exist a
set $A\le_T\emptyset'$ and sets
\[
B_0,B_1,\dots
\]
such that
\[
B_i\equiv_{\rm bfo}A
\qquad\text{for every }i\in\mathbb N,
\]
and
\[
B_i\le_1 B_j
\quad\text{if and only if}\quad
i\preceq j.
\]
Hence the function
\[
i\longmapsto \deg_1(B_i)
\]
is an order-embedding of $(\mathbb N,\preceq)$ into the partially ordered set of
$1$-degrees contained in the bounded finite-one degree of $A$.
\end{corollary}

\begin{proof}
Fix rationals $1<\alpha<\beta<2$, and for each $n\in\mathbb N$ let
\[
Z_n=\{2^n(2t+1):t\in\mathbb N\}.
\]
Then the sets $Z_n$ are pairwise disjoint, infinite, and uniformly computable.

For each $i,k\in\mathbb N$, define
\[
u_i(k)=
\begin{cases}
\beta & \text{if }k\in Z_n\text{ for some }n\preceq i,\\
\alpha & \text{otherwise.}
\end{cases}
\]
Since $\preceq$ is computable and the sets \(Z_n\) form a uniformly computable
pairwise disjoint cover of \(\mathbb N\setminus\{0\}\), the family \(\{u_i\}_{i\in\mathbb N}\) is uniformly
computable.

Let
\[
D=\{(i,j)\in\mathbb N^2 : i\npreceq j\}.
\]
Since $\preceq$ is computable, so is $D$. For each $(i,j)\in D$, define
\[
\delta_{i,j}=\beta-\alpha
\qquad\text{and}\qquad
s_{i,j}(t)=2^i(2t+1).
\]
Then $s_{i,j}$ is strictly increasing, computable, and has infinite range
\[
S_{i,j}=Z_i.
\]
Now fix $(i,j)\in D$, and let $k\in S_{i,j}=Z_i$. Since $\preceq$ is reflexive,
we have $i\preceq i$, so
\[
u_i(k)=\beta.
\]
If \(u_j(k)=\beta\), then \(k\in Z_n\) for some \(n\preceq j\). Since also
\(k\in Z_i\) and the sets \(Z_n\) are pairwise disjoint, it would follow that
\(n=i\), hence \(i\preceq j\), a contradiction. Therefore \(u_j(k)=\alpha\).
It follows that
\[
u_i(k)=\beta=\alpha+(\beta-\alpha)=u_j(k)+\delta_{i,j},
\]
so condition~\eqref{eq:profile-gap} holds.

Apply Theorem~\ref{thm:profile} and write
\[
B_i=C_i
\qquad(i\in\mathbb N).
\]
Then
\[
B_i\equiv_{\rm bfo}A
\qquad\text{for every }i\in\mathbb N.
\]

If $i\preceq j$, then for every $k\in\mathbb N$ we have $u_i(k)\le u_j(k)$:
indeed, if $u_i(k)=\alpha$ there is nothing to prove, while if $u_i(k)=\beta$,
then $k\in Z_n$ for some $n\preceq i$, and by transitivity $n\preceq j$, so
$u_j(k)=\beta$. Hence, by item~\ref{itm:profile-positive} of
Theorem~\ref{thm:profile},
\[
B_i\le_1 B_j.
\]

Conversely, if $i\npreceq j$, then $(i,j)\in D$, and item~\ref{itm:profile-negative}
of Theorem~\ref{thm:profile} gives
\[
B_i\not\le_1 B_j.
\]

Therefore
\[
B_i\le_1 B_j
\qquad\text{if and only if}\qquad
i\preceq j.
\]
Since each \(B_i\equiv_{\rm bfo}A\), the corresponding \(1\)-degrees all lie inside
the bounded finite-one degree of \(A\). Therefore
\[
i\longmapsto \deg_1(B_i)
\]
is an order-embedding of \((\mathbb N,\preceq)\) into the partially ordered set of
\(1\)-degrees contained in the bounded finite-one degree of \(A\).

\end{proof}

The previous corollary gives effective embeddings of computable partial orders
into the \(1\)-degrees inside a single bounded finite-one degree. By combining
that result with Mostowski's theorem on the existence of a computable universal
partial order, we immediately obtain a noneffective universality statement for
all countable partial orders. The only additional point is to ensure that the
witnessing bounded finite-one degree can still be chosen nonrecursive.

\begin{corollary}[Universal countable partial orders]\label{cor:profile-universal-poset}
There exists a nonrecursive set $A\le_T\emptyset'$ such that every countable
partial order embeds into the partially ordered set of $1$-degrees contained in
the bounded finite-one degree of $A$.
\end{corollary}

\begin{proof}

Let $\preceq_u$ be a computable universal partial order on $\mathbb N$, that is, a computable partial order into which every countable partial order embeds.
The existence of such an order is due to Mostowski \cite{Mostowski1938}; see also \cite[p.~11]{RSZ2026} and \cite[Exercise~V.2.8, p.~461]{OdifreddiCRT}.
Apply Corollary~\ref{cor:profile-computable-posets} to $\preceq_u$. Then there
exist a set $A\le_T\emptyset'$ and sets
\[
B_0,B_1,\dots
\]
such that
\[
B_i\equiv_{\rm bfo}A
\qquad\text{for every }i\in\mathbb N,
\]
and
\[
B_i\le_1 B_j
\qquad\text{if and only if}\qquad
i\preceq_u j.
\]
Every countable partial order thus embeds into
\[
\{\deg_1(B_i):i\in\mathbb N\},
\]
and hence into the partially ordered set of $1$-degrees contained in the
bounded finite-one degree of $A$.

It remains to show that $A$ may be taken nonrecursive. In the construction
from the proof of Theorem~\ref{thm:profile}, fix an index $e_{\mathrm{id}}$
for the identity function.
Since $\preceq_u$ is universal, it contains an
infinite antichain, so the set
\[
D=\{(i,j)\in\mathbb N^2 : i\npreceq_u j\}
\]
is infinite. For each $(i,j)\in D$, the requirement
\[
R_{e_{\mathrm{id}},i,j}
\]
reaches the diagonalization case, since the identity function is total and
injective on every finite set. Hence infinitely many stages assign one fresh
block the value $1$ and a later fresh block the value $0$. Therefore $A$ has
infinitely many $1$-blocks and infinitely many $0$-blocks, so $A$ is infinite
and coinfinite.

For each \(i\) and each \(k\), the block \(K_{i,k}\) is nonempty and
\[
h_i[K_{i,k}]=I_k.
\]
Since \(A\) is constant on \(I_k\), the set \(B_i=A\circ h_i\) is constant on
\(K_{i,k}\), with the same value that \(A\) takes on \(I_k\). As \(A\) has
infinitely many \(1\)-blocks and infinitely many \(0\)-blocks, it follows that
each \(B_i\) is infinite and coinfinite.

Since $\preceq_u$ is universal, it contains a
strict $2$-chain, so there exist $p,q\in\mathbb N$ such that
\[
p\preceq_u q
\qquad\text{and}\qquad
q\npreceq_u p.
\]
Hence
\[
B_p<_1 B_q.
\]
If $A$ were recursive, then every $B_i$ would be recursive. But any two
infinite coinfinite recursive sets are one-one equivalent. This contradicts
$B_p<_1 B_q$. Therefore $A$ is nonrecursive.

\end{proof}

\section{Concluding Remarks}

The results of this paper show that bounded finite-one degrees can support a
much richer internal order structure than previously known explicit
constructions suggested. In particular, our profile method yields a
nonrecursive set \(A\le_T\emptyset'\) whose bounded finite-one degree contains a
copy of the dense linear order \((\mathbb Q,\le)\), an infinite antichain of
\(1\)-degrees, and, more generally, an embedded copy of every countable partial
order. Thus a single bounded finite-one degree may simultaneously exhibit
dense, incomparable, and universal behaviour at the level of \(1\)-degrees.

This picture should be viewed against two complementary background results.
On the one hand, Richter, Stephan, and Zhang \cite{RSZ2026} constructed a
nonrecursive bounded finite-one degree whose internal \(1\)-degrees form exactly
a strictly ascending chain of order type \(\omega\), and asked whether there are
bounded finite-one degrees consisting exactly of a dense linearly ordered set of
\(1\)-degrees.
On the other hand, previous work on \(m\)-rigidity
\cite{Cintioli2026} shows that, for a measure-\(1\) and comeager class of sets,
bounded finite-one degrees already contain infinite antichains of
\(1\)-degrees and therefore cannot be linearly ordered. The present
construction complements that typical negative picture from the opposite
direction: it provides an explicit non-\(m\)-rigid witness, necessarily lying
in the null, meager complement of the typical rigid regime, whose bounded
finite-one degree is far from rigid and already contains a copy of
\((\mathbb Q,\le)\).

At the same time, our results do not settle the exact realization problem posed
by Richter, Stephan, and Zhang. The bounded finite-one degree constructed here
does not consist \emph{exactly} of a dense linear order of \(1\)-degrees, since
it also contains incomparable elements. What the profile method shows is that
the real difficulty of the problem lies not in producing density itself, but in
isolating a degree whose entire internal \(1\)-degree structure is precisely a
dense chain.
Thus, in the sense of embedded suborders, bounded finite-one degrees can already
realize every possible linear order.
What remains open, in the sense emphasized by Richter, Stephan, and Zhang in
their further remarks, is whether some bounded finite-one degree can have as its
entire internal \(1\)-degree structure a linear order other than the one-element
order or \(\omega\).
More broadly, the Profile Theorem isolates a flexible block-geometric
mechanism for controlling one-one comparison patterns inside a single bounded
finite-one degree.
The coexistence of constant and oscillating profiles, as well as the passage from
concrete examples to embeddings of arbitrary countable partial orders, shows that,
at the level of embedded suborders, the internal order structure of bounded
finite-one degrees is considerably richer than had previously been exhibited by
explicit constructions.

\section*{Acknowledgments}

This work is the result of an extended human--AI collaboration. 
Several structural ideas and technical arguments emerged from exploratory interaction with AI-based reasoning systems (Gemini Deep Think (Google DeepMind) and ChatGPT Pro (OpenAI)), which were used at different stages of the conceptual development and technical verification of this work. 
The author has fully reworked and verified all arguments and bears sole responsibility for their correctness.

\end{document}